\documentclass[12pt,reqno]{amsart}
\usepackage{mathtools,amssymb,amsmath,amsthm,enumerate,verbatim,bbm,bm}

\usepackage[a4paper]{geometry}

\DeclareMathOperator{\Ker}{Ker}

\DeclareMathOperator{\supp}{supp}

\newcommand{\abs}[1]{\lvert#1\rvert}

\newcommand{\norm}[1]{\lVert#1\rVert}

\newcommand{\Norm}[1]{\left\lVert#1\right\rVert}

\newcommand{\ls}{\le C}

\newcommand{\bbR}{{\mathbb R}}

\newcommand{\bbN}{{\mathbb N}}
\newcommand{\bbZ}{{\mathbb Z}}

\newcommand{\wh}{\widehat}

\newcommand{\ba}{{\mathbf a}}

\newcommand{\bb}{{\mathbf b}}
\newcommand{\bM}{{\mathbf M}}
\newcommand{\bH}{{\mathbf H}}
\newcommand{\bw}{{\mathbf w}}
\newcommand{\bk}{{\mathbf k}}
\newcommand{\bz}{{\bm{\zeta}}}

\newcommand{\calH}{{\mathcal H}}

\newcommand{\calL}{\mathcal{L}}

\newcommand{\calN}{\mathcal{N}}

\newcommand{\Sch}{\mathbf{S}}

\DeclareFontFamily{U}{mathx}{\hyphenchar\font45}
\DeclareFontShape{U}{mathx}{m}{n}{<5> <6> <7> <8> <9> <10>
<10.95> <12> <14.4> <17.28> <20.74> <24.88> mathx10}{}
\DeclareSymbolFont{mathx}{U}{mathx}{m}{n}
\DeclareFontSubstitution{U}{mathx}{m}{n}
\DeclareMathAccent{\widecheck}{0}{mathx}{"71}

\numberwithin{equation}{section}

\theoremstyle{plain}
\newtheorem{theorem}{\bf Theorem}[section]
\newtheorem*{theoremA}{Theorem A}

\newtheorem{lemma}[theorem]{\bf Lemma}

\theoremstyle{definition}

\theoremstyle{remark}
\newtheorem*{remark*}{\bf Remark}

\newcommand{\wt}{\widetilde}

\newcommand{\loc}{\mathrm{loc}}

\newcommand{\1}{\mathbbm{1}}

\newcommand{\bh}{{\mathbf{h}}}
\newcommand{\bbf}{{\mathbf{f}}}
\renewcommand{\[}{\begin{equation}}

\renewcommand{\]}{\end{equation}}

\begin{document}
\title{A Helson matrix with explicit eigenvalue asymptotics}
\date{\today}

\author{Nazar Miheisi \and Alexander Pushnitski}

\address{Department of Mathematics,
King's College London,
Strand, London WC2R 2LS,
United Kingdom}

\email{nazar.miheisi@kcl.ac.uk}

\email{alexander.pushnitski@kcl.ac.uk}



\begin{abstract}
A Helson matrix (also known as a multiplicative Hankel matrix) is an
infinite matrix with entries $\{a(jk)\}$ for $j,k\geq1$. Here the $(j,k)$'th term depends
on the product $jk$. We study a self-adjoint Helson matrix for a particular sequence
$a(j)=(\sqrt{j}\log j(\log\log j)^\alpha))^{-1}$,
$j\geq 3$, where $\alpha>0$,
and prove that it is compact and that its eigenvalues obey the asymptotics
$\lambda_n\sim\varkappa(\alpha)/n^\alpha$ as $n\to\infty$, with an explicit
constant $\varkappa(\alpha)$.
We also establish some intermediate results (of an independent interest)
which give a connection between the spectral properties of a Helson matrix and
those of its continuous analogue, which we call the integral Helson operator.
\end{abstract}

\maketitle

\section{Introduction}\label{sec.a}

\subsection{Background: Hankel matrices}
We start our discussion by recalling relevant facts from the theory of Hankel matrices.
Let $\{b(j)\}_{j=0}^\infty$ be a sequence of complex numbers.
A \emph{Hankel matrix} is an infinite matrix of the form
$$
H(b)=\{b(j+k)\}_{j,k=0}^\infty,
$$
considered as a linear operator in $\ell^2(\bbZ_+)$, $\bbZ_+=\{0,1,2,\dots\}$.
One of the key examples of Hankel matrices is the \emph{Hilbert matrix},
which corresponds to the choice $b(j)=1/(j+1)$.
It is well known that the Hilbert matrix is bounded (but not compact).
From the boundedness of the Hilbert matrix by a simple argument one obtains
$$
b(j)=o(1/j), \quad j\to\infty \quad \Rightarrow \quad H(b) \text{ is compact.}
$$
A natural family of compact self-adjoint Hankel operators of this class was
considered in \cite{PY1}.
To state this result, we need some notation.
For a compact self-adjoint operator $A$, let us denote by $\{\lambda_n^+(A)\}_{n=1}^\infty$ the
non-increasing sequence of positive eigenvalues (enumerated with multiplicities taken into account),
and let $\lambda_n^-(A)=\lambda_n^+(-A)$.

\begin{theoremA}\cite{PY1}
Let $b(j)$ be a sequence of real numbers defined by
$$
b(j)=1/(j(\log j)^\alpha),\quad j\geq2;
$$
the choice of $b(0)$ and $b(1)$ (or of any finite number of $b(j)$) is not important.
Then the eigenvalues of the Hankel  matrix $H(b)$ have the asymptotic behaviour
\[
\lambda_n^+(H(b))=\frac{\varkappa(\alpha)}{n^\alpha}+o(n^{-\alpha}), \quad
\lambda_n^-(H(b))=O(n^{-\alpha-1}),
\quad n\to\infty,
\label{a0}
\]
where $\varkappa(\alpha)$ is an explicit coefficient:
\[
\varkappa(\alpha)=2^{-\alpha}\pi^{1-2\alpha}B(\tfrac1{2\alpha},\tfrac12)^\alpha,
\label{a1}
\]
and $B(\cdot,\cdot)$ is the standard Beta function.
\end{theoremA}

\begin{remark*}
For negative eigenvalues, this result is stated in a slightly weaker form in \cite{PY1}:
$\lambda_n^-(H(b))=o(n^{-\alpha})$. However, following the logic of the proof of our main
result below, it is easy to see
that in fact the estimate $O(n^{-\alpha-1})$ is valid in Theorem~A.
\end{remark*}

\subsection{Helson matrices}
In this paper, we consider an analogous question in the class of Helson matrices
(also known as multiplicative Hankel matrices). These are infinite matrices of the form
$$
M(a)=\{a(jk)\}_{j,k=1}^\infty,
$$
considered as linear operators in $\ell^2(\bbN)$. Here the $(j,k)$'th entry depends on the product
of indices $jk$ rather than on the sum $j+k$.
Helson matrices are a natural object in the theory of Hardy spaces of Dirichlet series, in the same way
as Hankel matrices are naturally related to the theory of classical Hardy spaces.
The study of Helson matrices was initiated in the pioneering paper  \cite{Helson};
see also the book \cite{QQ}   and a recent survey \cite{PerPu}.

The \emph{multiplicative Hilbert matrix} is a Helson matrix corresponding to the sequence
$$
a(j)=1/(\sqrt{j}\log j), \quad j\geq2
$$
(there are variants of this definition, see \cite{PerPu2}; this notion has not become standardised yet).
It is bounded and not compact, and its spectral properties are fully analogous
to the classical Hilbert matrix, see \cite{BPSSV,PerPu2}. Similarly to the Hankel case, it is not difficult to
see that
$$
a(j)=o(1/(\sqrt{j}\log j)), \quad j\to\infty \quad \Rightarrow \quad M(a) \text{ is compact.}
$$
In this paper, we consider a family of compact modifications of the multiplicative Hilbert matrix.
Our main result is
\begin{theorem}\label{thm.a1}
Let $\alpha>0$, and let $a(j)$ be the sequence of real numbers given by
$$
a(j)=1/(\sqrt{j}\log j(\log\log j)^\alpha)
$$
for all sufficiently large $j$ (the choice of finitely many values $a(j)$ is not important).
Then the Helson matrix $M(a)$ is compact and its sequence of eigenvalues
obeys the asymptotics
\[
\lambda_n^+(M(a))=\frac{\varkappa(\alpha)}{n^\alpha}+o(n^{-\alpha}),
\quad
\lambda_n^-(M(a))=O(n^{-\alpha-1}),
\quad n\to\infty,
\label{a2}
\]
where $\varkappa(\alpha)$ is given by \eqref{a1}.
\end{theorem}
Thus, we have a natural family of Helson matrices $M(a^{(\alpha)})$, parameterised by $\alpha$, 
such that $M(a^{(\alpha)})\in\Sch_p$ if and only if $p>1/\alpha$. 
Here $\Sch_p$ is the standard Schatten class, see Section~\ref{sec.a6} below.

Below we describe the key ideas of the proof of Theorem~\ref{thm.a1}; some of them
may be of an independent interest.
In order to do this, we need some definitions.

\subsection{Integral Hankel and Helson operators}
First we recall the definition of a classical object: integral Hankel operators.
For a complex valued \emph{kernel function}, or more generally a
distribution, $\bb$ on $\bbR_+$, we denote by
$\bH(\bb)$ the integral Hankel operator in $L^2(\bbR_+)$, formally defined by
$$
\bH(\bb): f\mapsto \int_0^\infty \bb(x+y)f(y)dy.
$$
Clearly, integral Hankel operators are continuous analogues of Hankel matrices.
Below we only consider bounded and compact Hankel operators.
We use boldface font to denote integral operators (and their kernels).

Next, for a complex valued function or distribution
$\ba$ on $(1,\infty)$,
let us consider an integral operator in $L^2(1,\infty)$, defined by
$$
\bM(\ba): f\mapsto \int_1^\infty \ba(ts)f(s)ds, \quad t\geq1.
$$
It will be convenient to call $\bM(\ba)$ an \emph{integral Helson operator}
(this is not a standard term). We regard $\bM(\ba)$ as a continuous analogue
of the Helson matrix $M(a)$.

Observe that by an exponential change of variables, $\bM(\ba)$ reduces to an integral Hankel operator.
More precisely, let $V$ be the unitary operator
\[
V: L^2(\bbR_+)\to L^2(1,\infty),
\quad
(Vf)(t)=\frac1{\sqrt{t}}f(\log t),
\quad t>1,
\label{d0}
\]
then
\[
V^*\bM(\ba)V=\bH(\bb), \quad \bb(x)=\ba(e^x)e^{x/2}, \quad x>0.
\label{aa9}
\]
Spectral theory of integral Hankel operators is very well developed, and below
we will use some available results for eigenvalue estimates and asymptotics of such operators
to deduce the corresponding statements for integral Helson operators.

Note that although $\bM(\ba)$ can be reduced to an integral Hankel operator
through the exponential change of variable $t=e^x$, no such ``change of variable"
exists on integers, and therefore in general there is no simple reduction of
Helson matrices to Hankel matrices.

\subsection{The strategy of the proof}

Consider the integral Helson operator $\bM(\ba)$ with the kernel function $\ba\in C^\infty([1,\infty))$ which
satisfies
\[
\ba(t)=t^{-1/2}(\log t)^{-1}(\log \log t)^{-\alpha}, \quad t\geq t_0>e.
\label{a9}
\]
Clearly, the sequence $a$ of Theorem~\ref{thm.a1} is the restriction of the function $\ba$ onto
$\bbN$ (up to finitely many terms). 
It will be convenient to have some notation for the operation of restriction onto integers.
If $\ba$ is a continuous function on $(1,\infty)$, let $r(\ba)$ denote the
sequence
\[
r(\ba)(j)=
\begin{cases}
0,\quad j=1, \\
\ba(j), \quad j\ge 2.
\end{cases}
\label{a12}
\]

\emph{We will prove that the operators $\bM(\ba)$ and $M(r(\ba))$ have the same leading order
asymptotics of both positive and negative eigenvalues.}
This reduces the question to the spectral analysis of $\bM(\ba)$. Further, as already discussed,
relation \eqref{aa9} reduces the spectral analysis of $\bM(\ba)$ to that of the integral Hankel
operator $\bH(\bb)$ with
\[
\bb(x)=e^{x/2}\ba(e^x)=x^{-1}(\log x)^{-\alpha}, \quad x\geq x_0>1.
\label{a8}
\]
This two step reduction procedure can be illustrated by the diagram
\[
M(r(\ba))\quad\to\quad\bM(\ba)\quad\to\quad\bH(\bb).
\label{a7}
\]
The integral operator $\bH(\bb)$ is a continuous analogue of the Hankel
matrix in Theorem~A. The eigenvalues of $\bH(\bb)$ satisfy the same
asymptotic relation as \eqref{a0}, i.e.
\[
\lambda_n^+(\bH(\bb))=\frac{\varkappa(\alpha)}{n^\alpha}+o(n^{-\alpha}), \quad
\lambda_n^-(\bH(\bb))=O(n^{-\alpha-1}),
\quad n\to\infty,
\label{a11}
\]
where $\varkappa(\alpha)$ is the same as in \eqref{a1};
this is again a result of \cite{PY1}.
Thus, reduction \eqref{a7} together with \eqref{a11} will yield a proof of Theorem~\ref{thm.a1}.

\subsection{Further details and the structure of the paper}\label{sec.a5}
While the second reduction in \eqref{a7} is straightforward, the first reduction
is technically a little more involved; we proceed to explain it.
We split  the sequence $a$ into two terms
$$
a(j)=a_0(j)+a_1(j).
$$
Here $a_0$ is a sequence which has the same asymptotics as $a$, but is given
by a convenient integral representation; $a_1$ is the error term.
More precisely, let us describe the choice of $a_0$.

We use the fact that (see \cite{Erdelyi}) for any $0<c<1$,
one has the Laplace transform asymptotics
$$
\int_0^c \abs{\log \lambda}^{-\alpha}e^{-x\lambda}d\lambda=x^{-1}(\log x)^{-\alpha}
\bigl(1+O((\log x)^{-1})\bigr), \quad x\to\infty.
$$
Substituting $x=\log t$ and multiplying by $t^{-1/2}$, we obtain
$$
\int_0^c \abs{\log \lambda}^{-\alpha}t^{-\frac12-\lambda}d\lambda
=
t^{-1/2}(\log t)^{-1}(\log \log t)^{-\alpha}\bigl(1+O((\log\log t)^{-1})\bigr),
\quad t\to\infty.
$$

Now let $w(\lambda)=\abs{\log\lambda}^{-\alpha}\chi(\lambda)$, where
$\chi\in C^\infty(\bbR_+)$ is a non-negative function such that
$\chi(\lambda)=1$ for all sufficiently small $\lambda>0$ and $\chi(\lambda)=0$
for $\lambda\geq1$.
We set
\[
\ba_0(t)=\int_0^\infty t^{-\frac12-\lambda}w(\lambda)d\lambda, \quad
\ba_1(t)=\ba(t)-\ba_0(t), \quad
t>1,
\label{aa8}
\]
where the function $\ba$ is given by \eqref{a9}.
Then, by the above calculation,
$$
\ba_1(t)=O(t^{-1/2}(\log t)^{-1}(\log \log t)^{-\alpha-1}), \quad t\to\infty.
$$
Further, with the notation \eqref{a12}, we set $a_0=r(\ba_0)$ and $a_1=r(\ba_1)$.

 In Section~\ref{sec.c}, we will prove that $M(a_0)$ is unitarily equivalent to $\bM(\ba_0)$,
up to a negligible term, and as a consequence, the spectral asymptotics of these two operators
coincide to all orders. In fact, we will prove a more general statement (see Theorem~\ref{thm.cc1}):
if $\ba_0$ is given by the integral representation \eqref{aa8} then, for a fairly general class of weights $w$,
the Helson integral operator $\bM(\ba_0)$ is unitarily equivalent to the Helson matrix $M(r(\ba_0))$,
up to a negligible term.

In Section~\ref{sec.d}, we will reduce the spectral estimates for $M(a_1)$ to those for $\bM(\ba_1)$.
More precisely, in Theorem~\ref{thm.a2} 
\emph{we prove that the linear operator $\bM(\ba)\mapsto M(r(\ba))$ is bounded in
Schatten classes $\Sch_p$ for $0<p\leq1$, i.e. one has the estimate}
$$
\norm{M(r(\ba))}_{\Sch_p}\leq C_p\norm{\bM(\ba)}_{\Sch_p}, \quad 0<p\leq 1.
$$
This statement might be of an independent interest. 
By using real interpolation, we obtain the implication
$$
s_n(\bM(\ba))=O(n^{-\alpha-1}), \quad n\to\infty
\quad\Rightarrow\quad
s_n(M(r(\ba)))=O(n^{-\alpha-1}), \quad n\to\infty,
$$
for any $\alpha>0$, where $s_n$ are singular values
(see Section~\ref{sec.a6} below).

Thus, using somewhat different technical tools, we reduce the analysis
of both Helson matrices $M(a_0)$ and $M(a_1)$ to the corresponding integral
Helson operators $\bM(\ba_0)$ and $\bM(\ba_1)$. Next, we set, as in \eqref{aa9},
\[
\bb_i(x)=e^{x/2}\ba_i(e^x), \quad i=0,1,
\label{aa10}
\]
and use the available results from \cite{PY1,PY2} which give
\begin{align*}
\lambda_n^+(\bH(\bb_0))
&=
\varkappa(\alpha)n^{-\alpha}+o(n^{-\alpha}), \quad n\to\infty,
\\
s_n(\bH(\bb_1))
&=
O(n^{-\alpha-1}), \quad n\to\infty
\end{align*}
(we also have $\bH(\bb_0)\geq0$ and so $\lambda_n^-(\bH(\bb_0))=0$ for all $n$). 
Finally, in Section \ref{sec.e}, we use standard spectral stability
results to combine these two relations to
complete the proof of Theorem \ref{thm.a1}.

We represent this refined explanation of our proof by the following diagram:
\begin{gather*}
M(a)=M(a_0)+M(a_1);
\\
{\begin{split}
M(a_0)\to \bM(\ba_0)\to \bH(\bb_0)&\to \text{\cite{PY1}: asymptotics}
\\
M(a_1)\to \bM(\ba_1)\to \bH(\bb_1)&\to \text{\cite{PY2}: estimates}
\end{split}} \biggr\}\text{(stability)}
\Rightarrow \text{Theorem~\ref{thm.a1}.}
\end{gather*}

\subsection{Notation: Schatten classes}\label{sec.a6}
We denote by $\{s_n(A)\}_{n=1}^\infty$ the non-increasing sequence of the singular
values of a compact operator $A$, i.e. $s_n(A)=\lambda_n^+(\sqrt{A^*A})$.
Recall that for $0<p<\infty$, the Schatten class $\Sch_p$ consists of all compact operators $A$ such that
$$
\norm{A}_{\Sch_p}:=\left(\sum_{n=1}^\infty s_n(A)^p\right)^\frac{1}{p}
<\infty.
$$
We will write $\Sch_\infty$ to denote the class of compact operators.
Observe that $\norm{A}_{\Sch_p}$ is a norm for $p\geq1$ and a quasi-norm for $0<p<1$.
For $0<p<1$, the usual triangle inequality fails in $\Sch_p$  but the following ``modified triangle inequality" holds:
\[
\norm{A+B}_{\Sch_p}^p\leq \norm{A}_{\Sch_p}^p+\norm{B}_{\Sch_p}^p, \quad 0<p<1, \quad A,B\in \Sch_p.
\label{a10}
\]

For $0<p<\infty$ and $0<q\le \infty$, the Schatten-Lorentz class
$\Sch_{p,q}$ consists of all compact operators $A$ such that
\textbf{}
$$
\norm{A}_{\Sch_{p,q}}:=
\begin{dcases}
\left(\sum_{n=1}^\infty s_n(A)^q(1+n)^{q/p -1}\right)^{\frac{1}{q}}
<\infty, \quad q<\infty, \\
\sup_{n\in\bbN} (1+n)^{1/p}s_n(A)<\infty, \quad q=\infty.
\end{dcases}
$$
It is evident that $\Sch_{p,p} = \Sch_p$ for every
$0<p<\infty$. 
The classes $\Sch_{p,\infty}$ are known as weak
Schatten classes and have the property that $A\in \Sch_{p,\infty}$
if and only if $s_n(A)=O(n^{-1/p})$, $n\to\infty$.

We denote $\Sch_0=\cap_{p>0} \Sch_p$. This is the class of all operators $A$ such that $s_n(A)=O(n^{-c})$
as $n\to\infty$ for any $c>0$.

\subsection{Notation: unitary equivalence modulo kernels}
If $A_j$ is a bounded operator in a Hilbert space $\calH_j$ for $j=1,2$, we will
say that $A_1$ and $A_2$ are unitarily equivalent modulo kernels and write
$A_1\approx A_2$, if the operators
$$
A_1|_{(\Ker A_1)^\perp}
\quad\text{ and }\quad
A_2|_{(\Ker A_2)^\perp}
$$
are unitarily equivalent. It is well known that for any bounded operator $A$
(acting from a Hilbert space to a possibly different Hilbert space), one has
\[
A^*A\approx AA^*.
\label{b8}
\]
We will frequently use this relation in the following situation: if $A$ is compact,
then \eqref{b8} implies that $s_n(A^*A)=s_n(AA^*)$ for all $n$.

\subsection{Acknowledgements}
We are grateful to K.~Seip and H.~Queff\'elec for stimulating discussions, and to 
J.~Partington for help with the relevant literature. 

\section{$\bM(\ba)\approx M(r(\ba))$ up to error term}\label{sec.c}

\subsection{Overview}
In this section, we prove
\begin{theorem}\label{thm.cc1}
Let $w$ be a non-negative bounded function on $\bbR_+$ with bounded support.
Let
\[
\ba(t)=\int_0^\infty t^{-\frac12-\lambda}w(\lambda)d\lambda, \quad t>1,
\label{cc1}
\]
and let $a(j)=\ba(j)$, $j\in\bbN$. Then
we have $M(a)\geq0$ and $\bM(\ba)\geq0$. Further,
there exist self-adjoint operators $A$ and $B$ in
$L^2(\bbR_+)$ such that
$$
M(a)\approx A, \quad \bM(\ba)\approx B, \quad A-B\in \Sch_0.
$$
\end{theorem}

In combination with standard results on the stability of spectral asymptotics, Theorem~\ref{thm.cc1}
shows that if $\bM(\ba)$ and $M(a)$ are compact, then the eigenvalue
asymptotics of these operators coincide to all orders. This is precisely what we need
in our setting --- see Section~\ref{sec.e}.

Although our primary interest in this paper is to compact Helson matrices,
Theorem~\ref{thm.cc1} can be used in the non-compact context as well.
Indeed, in combination with the Weyl theorem on the invariance of the essential spectrum
with respect to compact perturbations, this result shows that
the non-zero parts of the essential spectra of $\bM(\ba)$ and $M(a)$ coincide.
Similarly, in combination with the Kato-Rosenblum theorem, this result shows that
the absolutely continuous parts of $\bM(\ba)$ and $M(a)$ are unitarily equivalent.
Variants of this reasoning have been used in \cite{BPSSV,PerPu2} in order to analyse
the multiplicative Hilbert matrix.

\subsection{Reduction to weighted integral Hankel operator}

We start by recalling a theorem from \cite{PerPu} which establishes a unitary
equivalence modulo kernels between a Helson matrix $M(a)$, where $a$ has
an integral representation of the type \eqref{cc1}, and a weighted integral Hankel type operator
$w^{1/2}\bH(\bz(\cdot+1))w^{1/2}$ with the integral kernel
$$
w(x)^{1/2}\bz(x+y+1)w(y)^{1/2}, \quad x,y>0
$$
in $L^2(\bbR_+)$, where $\bz$ is the Riemann zeta function.

\begin{lemma}\label{lma.c1}
Let $w\in L^\infty(\bbR)\cap L^1(\bbR)$ be a non-negative function, and let
$$
a(j)=\int_0^\infty j^{-\frac12-\lambda}w(\lambda)d\lambda, \quad j\geq1.
$$
Then the Helson matrix $M(a)$ is a bounded non-negative operator on $\ell^2(\bbN)$.
Let $\bz_1(x)=\bz(x+1)$. Then $w^{1/2}\bH(\bz_1)w^{1/2}$ is
bounded on $L^2(\bbR)$ and  
$$
M(a)\approx w^{1/2}\bH(\bz_1)w^{1/2}.
$$
\end{lemma}

\begin{proof}
This was proven in \cite{PerPu}, but for completeness we repeat the proof.

First let us check that $w^{1/2}\bH(\bz_1)w^{1/2}$ is bounded.
Recall that the Carleman operator $\bH(1/x)$, i.e. the integral Hankel operator with the
kernel function $\bb(x)=1/x$, is bounded on $L^2(\bbR_+)$ and has norm $\pi$.
Next, we have an elementary estimate
$$
0\leq \bz(x+1)-1=\sum_{j=2}^\infty j^{-x-1}\leq \int_1^\infty \frac{dt}{t^{x+1}}=\frac1x,
$$
and so for $\bb(x)=\bz(x+1)-1$, we obtain the estimate $\norm{\bH(\bb)}\leq \pi$.
Further, we have
$$
w^{1/2}\bH(\bz_1)w^{1/2}=w^{1/2}\bH(\bb)w^{1/2}+(\cdot,w^{1/2})w^{1/2},
$$
where the second term denotes the rank one operator in $L^2(\bbR_+)$
with the integral kernel $w(x)^{1/2}w(y)^{1/2}$.
Since $w\in L^1(\bbR)\cap L^\infty(\bbR)$, both terms here are bounded:
$$
\norm{w^{1/2}\bH(\bb)w^{1/2}}\leq \pi\norm{w^{1/2}}_{L^\infty}^2=\pi\norm{w}_{L^\infty},
\quad
\norm{(\cdot,w^{1/2})w^{1/2}}=\norm{w^{1/2}}_{L^2}^2=\norm{w}_{L^1}.
$$
We obtain that $w^{1/2}\bH(\bz_1)w^{1/2}$ is bounded.

Next, consider the operator
$$
\calN: L^2(\bbR_+)\to \ell^2(\bbN),
\quad
f\mapsto\biggl\{\int_0^\infty j^{-x-\frac12}w(x)^{1/2}f(x)dx\biggr\}_{j=1}^\infty,
$$
defined initially on the dense set of functions $f\in L^2(\bbR_+)$ with support
separated away from zero.
We claim that $\calN$ is bounded and $w^{1/2}\bH(\bz_1)w^{1/2}=\calN^*\calN$.
This is a direct calculation:
\begin{align*}
(\calN f_1,\calN f_2)_{\ell^2(\bbN)}
&=
\sum_{j=1}^\infty
\int_0^\infty \int_0^\infty j^{-1-x-y}w(x)^{1/2}w(y)^{1/2}f_1(x)\overline{f_2(y)}dx \, dy
\\
&=
\int_0^\infty \bz(x+y+1)w(x)^{1/2}w(y)^{1/2}f_1(x)\overline{f_2(y)}dx\, dy
\\
&=
(w^{1/2}\bH(\bz_1)w^{1/2} f_1,f_2)_{L^2(\bbR_+)},
\end{align*}
which proves our claim.

Further, let us compute the adjoint $\calN^*$:
$$
\calN^*: \ell^2(\bbN)\to L^2(\bbR_+), \quad
u=\{u_j\}_{j=1}^\infty \mapsto w(x)^{1/2}\sum_{j=1}^\infty u_j j^{-\frac12-x},
\quad
x>0.
$$
Then for $u,v\in \ell^2(\bbN)$ we have
\begin{multline*}
(\calN^* u,\calN^*v)_{L^2(\bbR_+)}
=
\int_0^\infty w(x) \biggl(\sum_{j,k=1}^\infty (jk)^{-\frac12-x}u_j\overline{v_k}\biggr)dx
\\
=
\sum_{j,k=1}^\infty a(jk) u_j\overline{v_k}
=
(M(a)u,v)_{\ell^2(\bbN)}.
\end{multline*}
This calculation proves that  $M(a)$ is bounded and $M(a)=\calN\calN^*$.

To summarise: for a bounded operator $\calN$, we have proven the identities
$$
w^{1/2}\bH(\bz_1)w^{1/2}=\calN^*\calN, \quad M(a)=\calN\calN^*.
$$
This shows that $w^{1/2}\bH(\bz_1)w^{1/2}\approx M(a)$, as required.
\end{proof}

\subsection{Reduction to a weighted Carleman operator}

\begin{lemma}\label{lma.cc4}
Let $w$ be as in Theorem~\ref{thm.cc1}.
Then
$$
w^{1/2}\bH(\bz_1)w^{1/2}
-
w^{1/2}\bH(1/x)w^{1/2}\in\Sch_0.
$$
\end{lemma}
\begin{proof}
We will need one well-known statement: if $\bb$ is a restriction of a
Schwartz class function onto $\bbR_+$, then the integral Hankel operator $\bH(\bb)$ is
in $\Sch_0$. This fact follows easily from Theorem~\ref{thm.b3} below.

\emph{Step 1:}
First we would like to replace $\bz(1+x)$ in the integral kernel of $w^{1/2}\bH(\bz_1)w^{1/2}$
by a simpler function $\bh$ with the same singularity at $x=0$.
Let $\beta >0$ be sufficiently large so that $\supp w\subset[0,\beta]$;  we choose
$$
\bh(x)=e^{-\beta x}/x, \quad x>0.
$$
Our aim at this step is to prove that the error term arising through this replacement is negligible, i.e. 
$$
w^{1/2}\bH(\bz_1)w^{1/2}-w^{1/2}\bH(\bh)w^{1/2}\in \Sch_0.
$$
Since the zeta function $\bz(z)$ has a simple pole at $z=1$ with residue one and 
converges to $1$ as $O(2^{-z})$ when $z\to+\infty$, we conclude
that the function
$$
\wt \bh(x)=\bz_1(x)-\bh(x)-1, \quad x>0,
$$
is a restriction of a Schwartz class function onto $\bbR_+$.
It follows that $\bH(\wt\bh)\in\Sch_0$.
Thus,
$$
w^{1/2}\bH(\bz_1)w^{1/2}-w^{1/2}\bH(\bh)w^{1/2}
=
w^{1/2}\bH(\wt \bh)w^{1/2}+(\cdot,w^{1/2})w^{1/2}\in\Sch_0;
$$
here the last term is the rank one operator with the integral kernel $w(x)^{1/2}w(y)^{1/2}$.

\emph{Step 2:}
Now it remains to prove that
\[
w^{1/2}\bH(1/x) w^{1/2}
-
w^{1/2}\bH(\bh) w^{1/2}
\in\Sch_0.
\label{cc3}
\]
Let $\calL$ be the Laplace transform in $L^2(\bbR_+)$,
$$
\calL [f](x)=\int_0^\infty e^{-x\lambda}f(\lambda)d\lambda.
$$
Observe that $1/x-\bh(x)=\calL[\1_{(0,\beta)}](x)$, where
$\1_{(0,\beta)}$ is the characteristic function of the interval $(0,\beta)$.
Thus, the operator in \eqref{cc3} can be written as
$w^{1/2}\calL \1_{(0,\beta)}\calL w^{1/2}$.

Since $w$ is bounded and $w^{1/2}=w^{1/2}\1_{(0,\beta)}$, it suffices to prove
that $\1_{(0,\beta)}\calL\1_{(0,\beta)}\in\Sch_0$.
Let $U$ be the unitary operator
$$
U: L^2(0,\beta)\to L^2(\bbR_+), \quad
Uf(x)=\sqrt{\beta}e^{-x/2}f(\beta e^{-x}), \quad x>0.
$$
A straightforward calculation shows that
$$
U\1_{(0,\beta)}\calL\1_{(0,\beta)} U^*=\bH(\bk),
$$
where the kernel function $\bk$ is given by
$$
\bk(x)=\beta e^{-x/2}\exp(-\beta^2e^{-x}), \quad x>0.
$$
Clearly, $\bk$ is a Schwartz class function (to be precise, a restriction
of a Schwartz class function onto the positive half-axis).
Thus,  $\bH(\bk)\in\Sch_0$ and so, by unitary equivalence,
we obtain $\1_{(0,\beta)}\calL\1_{(0,\beta)}\in\Sch_0$.
\end{proof}

\subsection{Reduction to $\bM(\ba)$ and completing the proof}
\begin{lemma}\label{lma.cc5}
Let $w\in L^\infty\cap L^1$. Then $w^{1/2}\bH(1/x)w^{1/2}\geq0$ and 
$$
w^{1/2}\bH(1/x)w^{1/2}\approx \bM(\ba),
$$
with $\ba$ as in \eqref{cc1}.
\end{lemma}
\begin{proof}
This argument is well known in the context of integral Hankel operators.
We have
$$
w^{1/2}\bH(1/x)w^{1/2}=w^{1/2}\calL \calL w^{1/2}=(w^{1/2}\calL)(w^{1/2}\calL)^*
\approx
(w^{1/2}\calL)^*(w^{1/2}\calL)
=\bH(\bb),
$$
where $\bb=\calL[w]$.
Now it remains to observe that, with $V$ as in \eqref{d0},  we have, by \eqref{aa9},
$$
V\bH(\bb)V^*=\bM(\ba),
$$
with
$$
\ba(t)=t^{-1/2}\bb(\log t)
=
t^{-1/2}\int_0^\infty e^{-\lambda\log t}w(\lambda) d\lambda
=
\int_0^\infty t^{-1/2-\lambda}w(\lambda)d\lambda,
$$
as required.
\end{proof}

\begin{proof}[Proof of Theorem~\ref{thm.cc1}]
Combining Lemmas~\ref{lma.c1}, \ref{lma.cc4} and \ref{lma.cc5}, we obtain
the required statement with $A=w^{1/2}\bH(\bz_1)w^{1/2}$ and $B=w^{1/2}\bH(1/x)w^{1/2}$.
\end{proof}

\section{The map $\bM(\ba)\mapsto M(r(\ba))$ is bounded in $\Sch_p$ for $p\leq1$}\label{sec.d}

\subsection{Overview}
Below for
$\bM(\ba)\in\Sch_p$, $0<p\leq1$, we will associate with $\ba$ its
restriction $r(\ba)$. In order for this restriction to make sense,
we need a preliminary statement, the proof of which is given in
Section~\ref{sec.d3}:

\begin{lemma}\label{lma.d0}
If $\bM(\ba)\in\Sch_1$, then the kernel function $\ba(t)$ is continuous
in $t>1$.
\end{lemma}

Before continuing, let us fix some notation:
throughout the remainder $C_p$ (or occasionaly $C'_p$) will denote
a constant which only depends on $p$ but whose precise value may change
from line to line.

Our main result in this section is:

\begin{theorem}\label{thm.a2}
\begin{enumerate}[\rm (i)]
\item
Assume that  $\bM(\ba)$ is bounded and belongs to the Schatten class $\Sch_p$
with $0<p\leq1$. Then $M(r(\ba))$ is also in $\Sch_p$, with the norm bound
\[
\norm{M(r(\ba))}_{\Sch_p}\ls_p \norm{\bM(\ba)}_{\Sch_p}.
\label{d00}
\]
\item
Assume that  $\bM(\ba)$ is bounded and belongs to the Schatten-Lorentz
class $\Sch_{p,q}$ with $0<p<1$ and $1\leq q\le\infty$.
Then $M(r(\ba))$ is also in $\Sch_{p,q}$, with the norm bound
$$
\norm{M(r(\ba))}_{\Sch_{p,q}}\ls_p \norm{\bM(\ba)}_{\Sch_{p,q}}.
$$
\end{enumerate}
\end{theorem}
We will only need the case $q=\infty$ of part (ii) of the theorem. 

In Sections~\ref{SP estimate}--\ref{sec.c4} after some preliminaries, we prove 
part (i) of the theorem. The proof uses V.~Peller's description of Hankel operators
of the class $\Sch_p$, $0<p\leq1$. 
In Sections~\ref{sec.c5}--\ref{sec.c6} we use real interpolation to deduce part (ii) 
of the theorem.

\subsection{Eigenvalue estimates for integral Hankel operators}\label{SP estimate}

For $\bbf\in L^1(\bbR)+L^2(\bbR)$, its Fourier transform is defined as usual by
\begin{equation*}
\wh\bbf(\xi):=\int_{-\infty}^\infty
\bbf(x)e^{-2\pi ix\xi}\,dx, \quad \xi\in\bbR.
\end{equation*}
Throughout this section, we let $\bw\in C^\infty_0(\bbR)$ be a function with
the properties $\bw\geq0$,  $\supp \bw=[1/2,2]$ and
$$
\sum_{n=-\infty}^\infty \bw(x/2^n)=1, \quad \text{for all } x>0.
$$
For $n\in\bbZ$, let $\bw_n(x)=\bw(x/2^n)$ and for a function
$\bb\in L^1_\loc(\bbR_+)$ set
\begin{equation}
\bb_n(x):=\bb(x)\bw_n(x), \quad x\in\bbR,
\label{d11}
\end{equation}
so that
$$
\wh{\bb}_n(\xi) = (\wh\bb\ast\wh\bw_n)(\xi), \quad \xi\in\bbR,
$$
where $*$ denotes convolution.
Clearly, we have
\[
\bb(x)=\sum_{n=-\infty}^\infty\bb_n(x), \quad x>0,
\label{c3a}
\]
where for every $x>0$, at most two terms of the series are non-zero.

Let us recall the necessary and sufficient conditions for
the Schatten class inclusion $\bH(\bb)\in \Sch_p$.

\begin{theorem}\cite[Theorem~6.7.4]{Peller}\label{thm.b3}
Let $\bb\in L^1_{\rm loc} (\bbR_{+})$ and let $p>0$.
The estimate
\begin{equation}
C_p\norm{\bH(\bb)}_{\Sch_p}^p
\le
\sum_{n=-\infty}^\infty 2^n \norm{\wh\bb_n}^p_{L^p(\bbR)}
\le C'_p
\norm{\bH(\bb)}_{\Sch_p}^p
\label{cb8}
\end{equation}
holds,
so that $\bH(\bb)\in \Sch_p$ if and only if  the series in
\eqref{cb8} converges.
\end{theorem}

The convergence of the series in \eqref{cb8} means that $\wh \bb$
belongs to the homogenous Besov class $B^{1/p}_{p,p}(\bbR)$.

\subsection{Preliminary statements}\label{sec.d3}

\begin{proof}[Proof of Lemma~\ref{lma.d0}]
Using the unitary equivalence \eqref{aa9} reduces the question to the following one:
if $\bH(\bb)\in\Sch_1$, then the kernel function $\bb(x)$ is continuous in $x>0$.
This statement is known, and the proof is evident:
if $\bb_n$ is as in \eqref{d11}, then 
by Theorem~\ref{thm.b3}, we have $\wh\bb_n\in L^1(\bbR)$ for all $n$, and so
in the series \eqref{c3a} each function $\bb_n$ is continuous. 
\end{proof}

A key ingredient of the proof of Theorem \ref{thm.a2} is a (scaled) classical
inequality of Plancherel and Polya \cite{Plan-Pol,Eoff} which states
that if $\bbf\in L^p(\bbR)$ for $p>0$  and $\supp\wh \bbf\subset[0,N]$, $N>0$,
then 
$$
\sum_{m=-\infty}^\infty\abs{\bbf(m/N)}^p  \ls_p N\norm{\bbf}^p_{L^p(\bbR)}.
$$
\begin{lemma}\label{lma.d5}
Let $v\in L^1(\bbR)\cap L^p(\bbR)$, $0<p\leq1$, and assume that the function
\[
\ba(t)=\int_{-\infty}^\infty v(\xi)t^{-\frac12+2\pi i \xi}d\xi,\quad t>0,
\label{d3}
\]
satisfies the condition $\supp\ba\subset[1,e^N]$ for some $N\in\bbN$.
Then for $a=r(\ba)$, the Helson matrix $M(a)$ satisfies the estimate
$$
\norm{M(a)}_{\Sch_p}^p\leq C_pN\norm{v}_{L^p(\bbR)}^p.
$$
\end{lemma}

\begin{proof}
Condition $\supp\ba\subset[1,e^N]$ means that we may regard $M(a)$ as
an $[e^N]\times [e^N]$ matrix ($[e^N]$ is the integer part of $e^N$);
we will use this throughout the proof.

1)
Let $p=1$.
Equation \eqref{d3} implies that
\[
a(jk)=\int_{-\infty}^\infty v(\xi) (jk)^{-\frac12+2\pi i \xi}d\xi.
\label{d5}
\]
This can be interpreted as an integral representation for $M(a)$ in terms of
rank one $[e^N]\times [e^N]$ matrices $\{(jk)^{-\frac12+2\pi i \xi}\}_{j,k=1}^{[e^N]}$.
The trace norm of these rank one matrices  is easy to compute:
$$
\Norm{\{(jk)^{-\frac12+2\pi i \xi}\}_{j,k=1}^{[e^N]}}_{\Sch_1}
=
\sum_{j=1}^{[e^N]}\abs{j^{-\frac12+2\pi i \xi}}^2
=
\sum_{j=1}^{[e^N]}\frac1j\leq 1+N.
$$
Substituting this estimate into \eqref{d5}, we get
$$
\norm{M(a)}_{\Sch_1}
\leq
\int_{-\infty}^\infty \abs{v(\xi)}\Norm{\{(jk)^{-\frac12+2\pi i \xi}\}_{j,k=1}^{[e^N]}}_{\Sch_1}d\xi
\leq
(N+1)\norm{v}_{L^1(\bbR)}.
$$

2) Let $0<p<1$.
Since the triangle inequality in $\Sch_p$ is no longer valid in this case, we have to use the
modified triangle inequality \eqref{a10}.
This forces us to use sums instead of integrals in estimates.
In particular, we need a series representation substitute for \eqref{d5}.
We claim that $\ba$ can be represented as
\[
\ba(t)=\frac1N\sum_{m=-\infty}^\infty v(m/N)t^{-\frac12+2\pi i\frac{m}N}, \quad t>1,
\label{d6}
\]
where the series converges absolutely and satisfies
\[
\sum_{m=-\infty}^\infty \abs{v(m/N)}^p\leq C_pN\norm{v}_{L^p(\bbR)}^p.
\label{d7}
\]
In order to justify this, we set $\bb(x)=e^{x/2}\ba(e^{x})$; then \eqref{d3} means
that $\wh\bb=v$.
Since $\supp\bb\subset[0,N]$, we can expand $\bb$ in the orthonormal basis
$$
N^{-1/2}e^{2\pi ix \frac{m}N}, \quad m\in\bbZ
$$
in $L^2(0,N)$.
This yields
$$
\bb(x)=\frac1N\sum_{m=-\infty}^\infty e^{2\pi ix \frac{m}N}\int_0^N \bb(y)e^{-2\pi i y\frac{m}N}dy
=\frac1N\sum_{m=-\infty}^\infty e^{2\pi i x\frac{m}N}v(m/N).
$$
Changing the variable $x=\log t$ and coming back to $\ba(t)$, we obtain \eqref{d6}.
Since $\supp\bb\subset[0,N]$, we can apply the Plancherel-Polya inequality, which gives \eqref{d7}.
The same inequality with $p=1$ ensures the absolute convergence of the series in \eqref{d6}
and justifies the above calculation.

3)
The representation \eqref{d6} yields
$$
a(jk)=\frac1N\sum_{m=-\infty}^\infty v(m/N)(jk)^{-\frac12+2\pi i\frac{m}{N}}, \quad j,k\in\bbN.
$$
This is an expansion of $M(a)$ in a series of rank one operators. As on step 1 of the proof, we have
$$
\Norm{\{(jk)^{-\frac12+2\pi i\frac{m}{N}}\}_{j,k=1}^N}_{\Sch_p}\leq(N+1).
$$
Applying the modified triangle inequality for $\Sch_p$ and using \eqref{d7}, we get
$$
\norm{M(a)}_{\Sch_p}^p
\leq
N^{-p}\sum_{m=-\infty}^\infty\abs{v(m/N)}^p(N+1)^p\leq C_p N\norm{v}_{L^p(\bbR)}^p,
$$
as required.
\end{proof}

\subsection{Proof of Theorem \ref{thm.a2}(i)}\label{sec.c4}

Let $\bb(x)=e^{x/2}\ba(e^x)$, $\bb_n(x)=\bb(x) \bw_n(x)$ and
$\ba_n(t)=t^{-1/2}\bb_n(\log t)$, $n\in\bbZ$, where $\bw_n$ are the functions
defined in Section~\ref{SP estimate}. Clearly, we have
\[
\ba(t)=\sum_{n=-\infty}^\infty \ba_n(t), \quad t>1.
\label{d9}
\]
From the unitary equivalence \eqref{aa9}, we see that
$\norm{\bM(\ba)}_{\Sch_p}=\norm{\bH(\bb)}_{\Sch_p}$.
Hence by Theorem \ref{thm.b3} we have
\[
\sum_{n=-\infty}^\infty 2^n\norm{\wh \bb_n}^p_{L^p(\bbR)}
\ls_p \norm{\bM(\ba)}_{\Sch_p}^p.
\label{estimate1}
\]

Fix $n\in\bbN$. We have $\supp\ba_n\subset[\exp(2^{n-1}),\exp(2^{n+1})]\subset[1,\exp(2^{n+1})]$,
and
$$
\ba_n(t)
=t^{-1/2}\bb_n(\log t)
=t^{-1/2}\int_{-\infty}^\infty \wh\bb_n(\xi)e^{i2\pi \xi\log t}d\xi
=\int_{-\infty}^\infty \wh\bb_n(\xi)t^{-\frac12+i2\pi \xi}d\xi.
$$
Also, by \eqref{estimate1} with $p=1$, we have $\wh \bb_n\in L^1(\bbR)$.
Thus, we can apply Lemma~\ref{lma.d5} with $N=2^{n+1}$, which yields
$$
\norm{M(r(\ba_n))}_{\Sch_p}^p \ls_p
2^n\norm{\wh \bb_n}^p_{L^p(\bbR)}.
$$
Now from \eqref{d9} we have
$$
M(r(\ba))=\sum_{n=-\infty}^\infty M(r(\ba_n));
$$
applying the modified triangle inequality \eqref{a10} for $\Sch_p$, we obtain
$$
\norm{M(r(\ba))}_{\Sch_p}^p
\leq
\sum_{n=-\infty}^\infty
\norm{M(r(\ba_n))}_{\Sch_p}^p
\leq
C_p
\sum_{n=-\infty}^\infty
2^n
\norm{\wh \bb_n}^p_{L^p(\bbR)}.
$$
Combining this with \eqref{estimate1}, we obtain the required estimate \eqref{d00}.
\qed

\subsection{Real interpolation}\label{sec.c5}

We now wish to show that the restriction map $\bM(\ba)\mapsto M(r(\ba))$
is bounded between the Schatten-Lorentz classes $\Sch_{p,q}$, when
$p<1$ and $1\leq q\le\infty$. To arrive at this we will use the real
interpolation method (the ``$K$-method''). We will quickly review
this, but refer the reader to \cite[\S 3.1]{Ber-Lof} for the details.

A pair of quasi-Banach spaces $(X_0,X_1)$ are called compatible if
they are both continuously included into the same Hausdorff
topological vector space. Real interpolation between a compatible
pair of quasi-Banach spaces $X_0$ and $X_1$ produces, for each
$0<\theta<1$ and $1\le q\le\infty$, an intermediate quasi-Banach
space which is denoted $(X_0,X_1)_{\theta,q}$ and which satisfies
$X_0\cap X_1\subseteq (X_0,X_1)_{\theta,q} \subseteq X_0 + X_1$,
with continuous inclusions.
In addition, if $(X_0,X_1)$ and $(Y_0,Y_1)$ are two pairs of
compatible quasi-Banach spaces and $A$ is a bounded linear map
from $X_0$ to $Y_0$ and from $X_1$ to $Y_1$ then $A$ will be
bounded from $(X_0,X_1)_{\theta,q}$ to $(Y_0,Y_1)_{\theta,q}$ for
each $0<\theta<1$ and $1\le q\le\infty$. 

An important result that we will make use of is the \emph{reiteration
theorem}: if $(X_0, X_1)$ are a compatible pair of quasi-Banach spaces,
then for $0\le\theta_0<\theta_1\le 1$ and $0<q_0,q_1<\infty$
\[
\left( (X_0, X_1)_{\theta_0,q_0}, (X_0, X_1)_{\theta_1,q_1}
\right)_{\theta,q}
= (X_0, X_1)_{(1-\theta)\theta_0 + \theta\theta_1,q},
\label{reiteration}
\]
where we interpret $(X_0,X_1)_{0,q}$ and $(X_0,X_1)_{1,q}$ to
be $X_0$ and $X_1$ respectively.

Of particular relevance
to us are the following interpolation spaces: for $0< p_0<p_1\le\infty$
and $0<q\le\infty$
\[
(\Sch_{p_0},\Sch_{p_1})_{\theta,q} = \Sch_{p,q},
\quad \frac{1}{p}=\frac{1-\theta}{p_0} + \frac{\theta}{p_1}.
\label{Sp interpolation}
\]

\subsection{Interpolation spaces of Hankel and Helson operators}\label{sec.c6}

Let $\bH\Sch_{p,q}$ denote the set of integral Hankel operators of
class $\Sch_{p,q}$, and let us write $\bH\Sch_p$ for $\bH\Sch_{p,p}$.
We claim that $\bH\Sch_{p,q}$ is a closed subspace of
$\Sch_{p,q}$ for all $0<p,q\le\infty$.
This is a straightforward consequence of the following characterisation
of integral Hankel operators \cite[Part B, Section 4.8, page 273]{Nikolski}. 
For $\lambda>0$, let
$S_\lambda$ denote the right shift by $\lambda$ on $L^2(\bbR_+)$
--- that is,
$$
S_\lambda:L^2(\bbR_+)\to L^2(\bbR_+), \quad
S_\lambda f(x) = 
\begin{cases}
f(x-\lambda), \quad x\ge\lambda, \\
0, \quad x<\lambda.
\end{cases}
$$
Then, for a bounded operator $A$ on $L^2(\bbR_+)$, one has
$A=\bH(\bb)$ for some distribution $\bb$ on $(0,\infty)$ if and only if 
$$
AS_\lambda = S_\lambda^*A \quad \text{for all}\quad \lambda>0.
$$

One has a description
for the interpolation spaces $(\bH\Sch_{p_0},\bH\Sch_\infty)_{\theta,q}$, see 
\cite[Theorem 6.4.1]{Peller}:
\[
(\bH\Sch_{p_0},\bH\Sch_\infty)_{\theta,q} = \bH\Sch_{p,q},
\quad p=\frac{p_0}{1-\theta}.
\label{HSp interpolation}
\]
It is worth noting that although \cite[Theorem 6.4.1]{Peller} is
stated for Hankel matrices, the same argument also works for
integral Hankel operators.

Similarly, let us write $\bM\Sch_{p,q}$ and $\bM\Sch_p$ to denote
the set of integral Helson operators of class $\Sch_{p,q}$ and $\Sch_p$
respectively. Since the unitary equivalence \eqref{aa9} provides
an isomorphism between $\bH\Sch_{p,q}$ and $\bM\Sch_{p,q}$ it
immediately follows from \eqref{HSp interpolation} that
\[
(\bM\Sch_{p_0},\bM\Sch_\infty)_{\theta,q} = \bM\Sch_{p,q},
\quad p=\frac{p_0}{1-\theta}.
\label{MSp interpolation}
\]

We are now in a position to conclude the proof of Theorem~\ref{thm.a2}.

\begin{proof}[Proof of Theorem~\ref{thm.a2}(ii)]

Fix $0<p_0<1$. Then by \eqref{MSp interpolation}, for any $p_1>p_0$,
we can write $\bM\Sch_{p_1} = (\bM\Sch_{p_0},\bM\Sch_\infty)_{\theta_1,p_1}$ for some $0<\theta_1<1$.
It then follows from the reiteration theorem \eqref{reiteration}
that
\[
(\bM\Sch_{p_0},\bM\Sch_{p_1})_{\theta,q}
= (\bM\Sch_{p_0},\bM\Sch_\infty)_{\theta\theta_1,q}
= \bM\Sch_{p,q},
\label{MSp interpolation 2}
\]
where $p$ is given by \eqref{Sp interpolation}.

By \eqref{d00}, the linear map $\bM(\ba)\mapsto M(r(\ba))$ is 
bounded from $\bM\Sch_{p_0}$ to $\Sch_{p_0}$ and from
$\bM\Sch_1$ to $\Sch_1$. It then follows from
\eqref{Sp interpolation} and \eqref{MSp interpolation 2}
that it is also bounded from $\bM\Sch_{p,q}$ to $\Sch_{p,q}$
for every $p_0<p<1$ and $1\leq q\le\infty$. This completes the proof.

\end{proof}

\section{Proof of Theorem~\ref{thm.a1}}\label{sec.e}

\subsection{Preliminaries}
Here we collect three results from other sources that will be needed below for the proof. 
The first one is the stability of the eigenvalue asymptotic coefficient, 
which is is standard in spectral perturbation theory.

\begin{lemma}\label{lma.b1}\cite[\S 11.6]{BSbook}
Let $A$ and $B$ be compact self-adjoint operators and let $\gamma>0$.
Suppose that
$s_n(A-B)=o(n^{-\gamma})$ as $n\to\infty$.
Then
\begin{align*}
\limsup_{n\to\infty}n^\gamma\lambda_n^+(A)
&=
\limsup_{n\to\infty}n^\gamma\lambda_n^+(B),
\\
\liminf_{n\to\infty}n^\gamma\lambda_n^+(A)
&=
\liminf_{n\to\infty}n^\gamma\lambda_n^+(B).
\end{align*}
\end{lemma}
Of course, similar relations hold true for negative eigenvalues $\lambda_n^-$.

Next, we need a result from \cite{PY1} on the spectral asymptotics of integral
Hankel operators. Roughly speaking, we need the eigenvalue asymptotics \eqref{a11} for 
integral Hankel operators $\bH(\bb)$ with the kernel as in \eqref{a8} --- this is one of the main 
results of \cite{PY1}. However, at the technical level, we need this result not for the kernel function $\bb$
of \eqref{a8}, but for the kernel function $\bb_0$ of \eqref{aa10}, which has the same asymptotics as $\bb$ but
is given by the suitable integral representation. This happens to be one of the intermediate 
results of \cite{PY1}, which fits our purpose. 
\begin{lemma}\cite[Lemma 3.2]{PY1}\label{lma.bb4}
Let $w(\lambda)=\abs{\log\lambda}^{-\alpha}\chi(\lambda)$, where
$\chi\in C^\infty(\bbR_+)$ is a non-negative function such that
$\chi(\lambda)=1$ near $\lambda=0$ and $\chi(\lambda)=0$
for $\lambda\geq1$. Consider the kernel function
$$
\bb_0(x)=\int_0^\infty w(\lambda)e^{-x\lambda}d\lambda,\quad x>0.
$$
Then the corresponding integral Hankel operator $\bH(\bb_0)$ is non-negative, compact
and has the spectral asymptotics
$$
\lambda_n^+(\bH(\bb_0))=\varkappa(\alpha)n^{-\alpha}+o(n^{-\alpha}), \quad n\to\infty.
$$
\end{lemma}

Finally, we will need a result from \cite{PY2} (which ultimately relies on Theorem~\ref{thm.b3}), 
which gives estimates on singular values for integral Hankel operators with kernels that
behave, roughly speaking, as $O(x^{-1}(\log x)^{-\gamma})$.  
For $\gamma>0$, denote
\begin{equation}
m(\gamma)=
\begin{cases}
[\gamma]+1& \text{ if } \gamma\geq1/2,
\\
0, & \text{ if } \gamma<1/2.
\end{cases}
\label{b7}
\end{equation}

\begin{theorem}\label{thm.b4}\cite[Theorem~2.7]{PY2}
Let $\gamma>0$ and  let $m=m(\gamma)$ be the integer given by \eqref{b7}.
Let $\bb$ be a complex valued function,
$\bb\in L^\infty_\loc(\bbR_+)$; if $\gamma\geq1/2$, suppose also
that $\bb\in C^m(\bbR_+)$.
Assume that $\bb$ satisfies
\[
\bb^{(\ell)}(x)=O(x^{-1-\ell}\abs{\log x}^{-\gamma})
\quad
\text{ as $x\to0$ and as $x\to\infty$,}
\label{b9}
\]
for all $\ell=0,\dots,m(\gamma)$.
Then
$$
s_n(\bH(\bb))=O(n^{-\gamma}), \quad n\to\infty.
$$
\end{theorem}

\subsection{Proof of Theorem~\ref{thm.a1}}

We use the notation of Section~\ref{sec.a5}.
More precisely, $\ba(t)$ is a smooth function that satisfies \eqref{a9} for large $t$ and
$\bb(x)$ is the corresponding Hankel kernel function \eqref{a8}.
Further,
$\chi\in C^\infty(\bbR_+)$ is a non-negative function such that
$\chi(\lambda)=1$ for all sufficiently small $\lambda>0$ and $\chi(\lambda)=0$
for $\lambda\geq1$, and $w(\lambda)=\abs{\log\lambda}^{-\alpha}\chi(\lambda)$.
The kernel functions $\ba_0$ and $\ba_1$ are given by \eqref{aa8} and the corresponding
Hankel kernels $\bb_0$, $\bb_1$ are given by \eqref{aa10}; finally, $a_0=r(\ba_0)$ and $a_1=r(\ba_1)$. 
Recall that we have 
$$
M(a)=M(a_0)+M(a_1).
$$

1) Let us prove that the Helson matrix $M(a_0)$ is compact, non-negative and
has the spectral asymptotics
\[
\lambda_n^+(M(a_0))=\varkappa(\alpha)n^{-\alpha}+o(n^{-\alpha}), \quad n\to\infty.
\label{dd0}
\]
Lemma~\ref{lma.bb4} provides the asymptotics of the required type for $\bH(\bb_0)$. 
By the unitary equivalence between $\bM(\ba_0)$ and $\bH(\bb_0)$, we have
$$
\lambda_n^+(\bM(\ba_0))=\lambda_n^+(\bH(\bb_0))
$$
for all $n$, and so $\bM(\ba_0)$ obeys the same spectral asymptotics. 
Finally, we use Theorem~\ref{thm.cc1} with $a=a_0$ and $\ba=\ba_0$.
By the unitary equivalence modulo kernels, we have
$$
\lambda_n^+(M(a_0))=\lambda_n^+(A),
\quad
\lambda_n^+(\bM(\ba_0))=\lambda_n^+(B),
$$
for all $n$, where $A$ and $B$ are the operators in the statement of Theorem~\ref{thm.cc1}. 
Now since $A-B\in\Sch_0$, by Lemma~\ref{lma.b1}, we have
$$
\limsup_{n\to\infty}n^\alpha\lambda_n^+(A)
=
\limsup_{n\to\infty}n^\alpha\lambda_n^+(B)
$$
for all $\alpha>0$, and similarly for $\liminf$.
This gives the required asymptotics for $\lambda_n^+(A)$ and so for $\lambda_n^+(M(a_0))$. 
The non-negativity $M(a_0)\geq0$ is given again by Theorem~\ref{thm.cc1}. 

2) Let us prove that 
the Helson matrix $M(a_1)$ is compact and satisfies the spectral estimates
\[
s_n(M(a_1))=O(n^{-\alpha-1}), \quad n\to\infty.
\label{dd1}
\]
By the choice of $\ba_1$, we have that $\bb_1$ is smooth on $[0,\infty)$  and
\[
\bb_1(x)=x^{-1}(\log x)^{-\alpha}-\int_0^\infty e^{-\lambda x}w(\lambda)d\lambda
\label{d1}
\]
for all sufficiently large $x$.
Let us check that $\bb_1$ satisfies the hypothesis of Theorem~\ref{thm.b4} with $\gamma=\alpha+1$.
Since $\bb_1$ is smooth near $x=0$, we only need to check \eqref{b9} for $x\to\infty$.

We use the following well known fact \cite{Erdelyi}.
Let $0<c<1$, $\ell\in\bbZ_+$, and
$$
I_\ell(x)= \int_{0}^c \abs{\log \lambda}^{-\gamma}\lambda^{\ell} e^{-\lambda x}d\lambda, \quad x>0.
$$
Then
$$
I_\ell(x)=\ell!\, x^{-1-\ell} (\log x)^{-\alpha} \bigl(1+O((\log x)^{-1})\bigr),
\quad x\to\infty.
$$
Now differentiating \eqref{d1} $\ell$ times, we obtain
\begin{multline*}
\bb_1^{(\ell)}(x)
=
(-1)^\ell
\ell!\, x^{-1-\ell} (\log x)^{-\alpha}
+O(x^{-1-\ell}(\log x)^{-\alpha-1})
\\
-
(-1)^\ell
\int_0^\infty e^{-\lambda x}\lambda^{\ell}w(\lambda)d\lambda
=
O(x^{-1-\ell}(\log x)^{-\alpha-1}),
\quad
x\to\infty,
\end{multline*}
for all $\ell\geq0$, which gives the required estimate \eqref{b9} with $\gamma=\alpha+1$.

Thus, Theorem~\ref{thm.b4} yields the inclusion $\bH(\bb_1)\in\Sch_{p,\infty}$ with $p=1/(\alpha+1)$.
By the unitary equivalence \eqref{aa9}, it follows that $\bM(\ba_1)\in\Sch_{p,\infty}$.
Applying Theorem~\ref{thm.a2}(ii), we obtain $M(a_1)\in\Sch_{p,\infty}$,
which is equivalent to the required estimate \eqref{dd1}.

3)
Now we can conclude the proof of the theorem. 
Let us apply the asymptotic stability result Lemma~\ref{lma.b1}
with $A=M(a_0)$ and $B=M(a_1)$.
By \eqref{dd0}, this gives the asymptotics \eqref{a2} for positive eigenvalues.

Let us discuss the estimate \eqref{a2} for negative eigenvalues.
By \eqref{dd1}, we have
\[
\lambda_n^-(M(a_1))=O(n^{-\alpha-1}),\quad n\to\infty.
\label{e1}
\]
Since $M(a_0)\geq0$, by the variational principle (see e.g. \cite[Theorem 9.3.7]{BSbook}), we obtain
$$
\#\{n: \lambda_n^-(M(a))>\lambda\}
\leq
\#\{n: \lambda_n^-(M(a_1))>\lambda\}
$$
for any $\lambda>0$, which implies
$$
\lambda_n^-(M(a))\leq \lambda_n^-(M(a_1))
$$
for any $n$. 
From here and \eqref{e1} we obtain the estimate \eqref{a2} for negative eigenvalues.
\qed

\end{document}